\documentclass[12pt]{amsart}
\usepackage{amsfonts}
\usepackage{amsmath}
\usepackage{amssymb}
\usepackage{graphicx}
\usepackage{amscd}
\usepackage{xcolor}
\usepackage[utf8]{inputenc}
\usepackage[greek,english]{babel}
\usepackage{alphabeta}
\usepackage{hyperref}
\newtheorem{theorem}{Theorem}

\newtheorem{lemma}[theorem]{Lemma}

\newtheorem*{proposition}{Proposition}

\begin{document}

		\begin{abstract}
		Let $\mathbb{T}$ be a homogeneous tree. We prove that if $f\in L^{p}(\mathbb{T})$, $1\leq p\leq 2$, then the Riesz means $S_{R}^{z}\left(
		f\right) $ converge to $f$ 
		everywhere as $R\rightarrow \infty $, whenever $\operatorname{Re}z>0 $. 
		\end{abstract}
		
		\title[Riesz means ]{Riesz means on homogeneous trees}
		\thanks{Supported by the Hellenic Foundation for Research and Innovation, Project HFRI-FM17-1733.}
		\author{Effie Papageorgiou}
		\curraddr{Department of Mathematics, Aristotle University of Thessaloniki,
			Thessaloniki 54.124, Greece }
		\maketitle
		
		A homogeneous tree $\mathbb{T}$ of degree $Q+1$, $Q\geq 2$, is an infinite connected graph with no loops, in which every vertex is adjacent to $Q+1$ other vertices. We shall identify $\mathbb{T}$ with its set of vertices $\mathcal{V}$. $\mathbb{T}$ carries a natural distance $d$ and a natural measure $\mu$. Specifically, $d(x,y)$ is the number of edges of the shortest path joining $x$ to $y$ and $\mu$ is the counting measure. For the counting measure, the volume of any sphere $S(x,n)$ in $\mathbb{T}$ is given by 
			\[
			|S(x,n)|=\begin{cases}
			1, & \text{if } n=0\\
			(Q+1)Q^{n-1}, & \text{if } n\in \mathbb{N},
			\end{cases}
			\]
			and the Lebesgue spaces, associated with  $\mu$, have norms defined by 
			$L^p(\mathbb{T})$ 
			\begin{equation}\label{Lp}
			\|f\|_{p}=
			\begin{cases}
			\Big( \int_{\mathbb{T}}|f(x)|^p d\mu \Big)^{1/p}=\Big( \sum\limits_{x\in\mathbb{T}}|f(x)|^p \Big)^{1/p}, & \text{if } 1 \leq p <\infty \\
			\sup\limits_{x\in \mathbb{T}} |f(x)|, & \text{if } p=\infty.%
			\end{cases}
			\end{equation} 
			
			Let us fix a base point $x_0$ and set $|x|=d(x,x_0)$. Functions depending only on $|x|$ are called radial.  If $E(\mathbb{T})$ is a function space on 
			$\mathbb{T}$, we will denote by $E(\mathbb{T})^{\#}$ the subspace of radial elements in $E(\mathbb{T})$.

			Let $G$ be the group of isometries of $\mathbb{T}$, and suppose that $G$ acts transitively on $\mathbb{T}$.  If $x_0\in \mathbb{T}$ is a
			fixed vertex, then the orbit $Gx_0$ is all of $\mathbb{T}$. Therefore $\mathbb{T}$ may be
			identified through the map $g\mapsto gx_0$ with the quotient $G/K$, where
			$K=\{g \in G: gx_0=x_0\}$, \cite[p. 46]{Figa}. This means that every function on $\mathbb{T}$ may be
			lifted to a function on $G$, by defining $\tilde{f}(g)=f(gx_0)$. The function
			$\tilde{f}$ has the property that $\tilde{f}(gk)=\tilde{f}(g)$, for every $k\in K$, and
			conversely a $K$-right-invariant function on $G$ may be identified
			with a function on $\mathbb{T}$. Thus, we shall identify
			functions defined on $\mathbb{T}$
			with right-$K$-invariant functions on $G$ and radial functions with $K$-bi-invariant functions on $G$.	
			
			Let $G$ be a locally compact group and $K$ be a compact subgroup of $G$; then the pair $(G,K)$ is called a Gelfand pair if the space $C_c(K\backslash G/K)$ of complex continuous $K$-bi-invariant functions with compact support is a commutative algebra with the convolution product (the space $C_c(K\backslash G/K)$ is always an algebra with the convolution product). Under the hypothesis that  $K$ acts transitively on the boundary of $\mathbb{T}$, $(G,K)$ is a Gelfand pair. Also, the transitive action of $K$ on the boundary of $\mathbb{T}$ is
			also a necessary condition for $(G,K)$ to be a Gelfand pair, \cite[p. 47]{Figa}. In fact, the following holds true, \cite{Tits}.
			
	\begin{proposition}
				For every finite subset $\mathcal{F}$ of  $\mathcal{V}$, denote by $\text{Aut}_{\mathcal{F}}(\mathbb{T})$ the group of automorphisms $g$ of $\mathbb{T}$ such that $g(x) = x$ for all $x \in \mathcal{F}$. Then $G$ is equipped with a topology of locally compact totally discontinuous group such that the subgroups $\text{Aut}_{\mathcal{F}}(\mathbb{T})$ form a fundamental system of neighborhoods of the identity in $G$. Moreover a subgroup $H$ is maximal, open, compact in $G$ if and only if $H$ is the stabiliser of a point $x$ in $\mathcal{V}$.
	\end{proposition}
					
		We normalize the Haar measure on $G$ in such a way that $K$ has a unit mass. Then
		\[\
		\sum\limits_{x\in \mathbb{T}}f(x)=\int_{G}f(g)dg, \; f\in L^1(\mathbb{T}).
		\]
		This allows us to define the convolution of two functions on $\mathbb{T}$ by
		\begin{equation} \label{11}
		(f_1\ast f_2)(g)=\int_{G}f_1(h)f_2(h^{-1}g)dh, \; g\in G. 
		\end{equation}
		
If $f_2$ is radial, then (\ref{11}) rewrites 
\begin{equation} \label{convrad}
(f_1\ast f_2)(x)=\sum\limits_{n\geq 0}f_2(n)\sum\limits_{y\in S(x,n)}f_1(y), \; x,y\in \mathbb{T}. 
\end{equation}

		In this short note, we deal with the Riesz means $S_R^z$,   $\operatorname{Re}z>0$, $R>0$, on $\mathbb{T}$, which are defined as convolution operators
		\[\
		S_R^zf=f\ast\kappa_{R}^z,
		\]
		where $\kappa_{R}^z$ is the inverse spherical transform of the multiplier
		\begin{equation}\label{mult}
		m_{R }^z(\lambda )=\left( 1-\frac{1-\gamma(\lambda)}{R} \right)^z_+  ,
		\end{equation}
		(see Section 1 for more details). 	We prove the following result.
		\begin{theorem}
			\label{Thm} Let $1\leq p\leq 2$. If  $\operatorname{Re}z>0$, then for $f\in L^{p}(\mathbb{T})$, 
			\begin{equation}
			\lim\limits_{R\rightarrow +\infty }S_{R}^{z}f(x)=f(x),\;\text{everywhere.}
			\end{equation}
		\end{theorem}
		Let us recall that the Riesz means were first treated by E.M.Stein in \cite{STEIN27}, where he proved that if $f\in L^{p}\left( \left[ 0,1\right] ^{n}\right) $, $n\geq 1$, $p\in \left( 1,2%
		\right] $, then
		\begin{equation}
		\left\Vert S_{R}^{z}\left( f\right) -f\right\Vert _{p}\longrightarrow 0,%
		\text{ as }R\longrightarrow \infty ,  \label{st}
		\end{equation}%
		whenever $\operatorname{Re}z>\left( \frac{n-1}{2}\right) \left( \frac{2}{p}-1\right) $. Since then, many authors have studied Riesz means in various geometric contexts, as euclidean spaces, 
		compact manifolds, Lie groups of
		polynomial volume growth, graphs and discrete groups of polynomial volume growth, Riemannian manifolds of nonnegative curvature, compact semisimple Lie
		groups and noncompact symmetric spaces. See \cite{ALEXOLO, CHRIST7, ChristSogge,  RieszFMP, RieszFG, GIUMA, MAR, SEEGER, Sogge, STEIN27,STEIN29}. 
		
		Note that the multiplier $m_{R}^{z}(\lambda )$ does not extend
		holomorphically to any strip containing the real line. So, by \cite[%
		Theorem 1.2]{COWMESE}, the Riesz means operator is not bounded on $L^{p}(\mathbb{T})$
		if $p\neq 2$ and consequently the norm summability problem on $L^{p}(\mathbb{T})$, $%
		p\neq 2$, is ill posed. Note also that in the case of homogeneous trees, there is no restriction on the size of $\operatorname{Re}z>0$, contrarily to the Euclidean case and even to the hyperbolic space case, due to absence of local obstructions. 
		
				\section{Preliminaries}
					In this section we present the tools we need for the proof of our results. For details, see for example \cite{COWMESE, EDDI}. Let $\mathcal{M}$ be the mean operator
					\begin{equation*}
					(\mathcal{M}f)(x)=\frac{1}{Q+1}\sum\limits_{y\in \mathbb{T},\; d(x,y)=1}f(y).
					\end{equation*} Then, the Laplacian $\mathcal{L}$ on $\mathbb{T}$ is defined by
				\[\
				\mathcal{L}=I-\mathcal{M}.
				\]
				The \textit{spherical function} $\varphi_\lambda$ of index $\lambda \in \mathbb{C}$ is the unique radial eigenfunction of the operator $\mathcal{L}$, which is associated with the eigenvalue 
				\begin{equation}\label{gammalambda0}
				\gamma(\lambda)=\frac{Q^{i\lambda}+Q^{-i\lambda}}{Q^{\frac{1}{2}}+Q^{-\frac{1}{2}}}=\frac{2}{Q^{\frac{1}{2}}+Q^{-\frac{1}{2}}}\cos(\lambda \log Q),
				\end{equation} and which is normalized by $\varphi_\lambda(0)=1$. Set $\tau=\frac{2\pi}{\log Q}$. Then,
				\begin{equation}\label{gammalambda}
				\gamma(\lambda)=\gamma(0)\cos(\frac{2\pi}{\tau}\lambda).
				\end{equation}
				Note that $\varphi_\lambda$ is periodic with period $\tau$.
		
				The \textit{spherical Fourier transform} of a radial function $f$ on $\mathbb{T}$ is defined by
				\[\
				(\mathcal{H}f)(\lambda)=\sum\limits_{x\in \mathbb{T}}f(x)\varphi_\lambda(x)=f(0)+\sum\limits_{n\geq 1}(1+Q)Q^{n-1}f(n)\varphi_\lambda(n), \quad  \lambda \in \mathbb{C}.
				\]
		The following inversion formula holds:
				\begin{equation} \label{inversion}
				(\mathcal{H}^{-1}f)(x)=\int_{-\tau /2}^{\tau /2}f(\lambda)\varphi_{\lambda}(x)\frac{d \lambda}{|\textbf{c}(\lambda)|^2}, \quad  x\in \mathbb{T},
				\end{equation}
				where $\textbf{c}$ is the meromorphic function 
				\begin{equation} \label{cfunc}
				\textbf{c}(z)=\frac{1}{Q^{1/2}+Q^{-1/2}} \frac{Q^{1/2+iz}-Q^{-1/2-iz}}{Q^{iz}-Q^{-iz}}, \quad  z \in \mathbb{C} \backslash (\frac{\tau}{2})\mathbb{Z}.	
				\end{equation}	
		
				We have the following Plancherel theorem \cite{COWMESE}: the spherical Fourier transform extends to an isometry of $L^2(\mathbb{T})^{\#}$ onto $L^2([-\tau/2,\tau/2], \frac{d \lambda}{|\textbf{c}(\lambda)|^2})$ and
\begin{equation}\label{plancherel}
\|f\|_{2}=\left( \int_{-\tau /2}^{\tau /2}|(\mathcal{H}f)(\lambda)|^2 \frac{d \lambda}{|\textbf{c}(\lambda)|^2} \right)^{1/2}.
\end{equation}

Note also that the spherical Fourier transform is written as a composition 
\[\
\mathcal{H}=\mathcal{F}\circ \mathcal{A}
\]
of the euclidean Fourier transform $\mathcal{F}$ and the \textit{Abel transform} $\mathcal{A}$ \cite{EDDI}. Recall that the kernel $\kappa_{R}^z$ of the Riesz means operator is given by the inverse spherical transform of the multiplier $m_R^z$:
\[\
\kappa_{R}^z=(\mathcal{F}\circ \mathcal{A})^{-1}(m_R^z)=(\mathcal{A}^{-1}\circ \mathcal{F}^{-1})(m_R^z).
\] For that, we shall make use of the following inversion formulas, \cite{EDDI}:
\begin{equation}\label{invFour}
(\mathcal{F}^{-1}f)(n)=\frac{1}{\tau}\int_{-\tau /2}^{\tau /2}f(\lambda)Q^{-i\lambda n}d\lambda=\frac{1}{\pi}\int_{0}^{\pi}f(\frac{\tau}{2\pi}\lambda)\cos (\lambda n)d\lambda,
\end{equation}
and 
\begin{equation}\label{invAbel}
(\mathcal{A}^{-1}f)(n)=\sum_{k=0}^{\infty}Q^{-\frac{n}{2}-k}\{ f(n+2k)-f(n+2k+2) \}. 
\end{equation}

\section{Estimates of the kernel $\kappa_R^z$}
\begin{lemma}\label{kernelest} If $\operatorname{Re}z>0$, then
	\[\
	|\kappa_R^z(n)|\leq c Q^{-n/2}.
	\]
\end{lemma}
\begin{proof}
We have that $\kappa_R^z=\mathcal{H}^{-1}(m_R^z)=(\mathcal{A}^{-1}\circ\mathcal{F}^{-1})(m_R^z)$. Then, using (\ref{invAbel}), (\ref{invFour}) and (\ref{mult}), we obtain the following explicit expression of $\kappa_{R}^z$:
\begin{align}\label{kernel1}
\kappa_R^z(n)&=\frac{1}{\pi}\sum_{k=0}^{\infty}Q^{-\frac{n}{2}-k}\int_{0}^{\pi}\left( 1-\frac{1-\gamma(\frac{\tau}{2\pi}\lambda)}{R} \right)_+^z\sin\lambda \sin(\lambda (n+2k+1))d\lambda\notag\\
&=\frac{1}{\pi}Q^{-\frac{n}{2}}\sum_{k=0}^{\infty}Q^{-k}\int_{0}^{\pi}\sin\lambda \sin(\lambda (n+2k+1))d\lambda\notag\\
&\leq cQ^{-\frac{n}{2}}\sum_{k=0}^{\infty}Q^{-k} \leq cQ^{-\frac{n}{2}}\notag.
\end{align}

\end{proof}
\textbf{Remark.}
	 Note that the same kernel estimate holds for every operator $m_R(\mathcal{L})$ with a uniformly bounded multiplier $|m_R(\lambda)|\leq c$.

\section{Proof of Theorem 1}
For the proof of Theorem 1, we need to introduce the maximal function associated with Riesz means:
\begin{equation}  \label{maximal}
S_*^z(f)(x)=\sup_{R>0 } |S^z_R(f)(x)|, \; f\in L^p(\mathbb{T}), \; 1\leq
p\leq 2.
\end{equation}
The proof will be given in steps.

\textit{Step 1: The maximal operator $S_*^z(f)$ is bounded from $L^p(\mathbb{T})$ to $L^r(\mathbb{T})$,  $1\leq p \leq 2$, and $r\geq
pq/(2-p+pq-q)$, for every $q>2$. }

\medskip
 First, proceeding as in \cite{GIUMA}, we have that
\begin{lemma}
	\label{S*L2}  If $f\in L^2(\mathbb{T})$ and $z\in \mathbb{C}$ with $\operatorname{Re}z>0$,
	then $\|S_*^zf\|_2\leq c(z)\|f\|_2$.
	\end{lemma}
\begin{proof}
	We include the proof for the sake of completeness.	Let $h_R$, $R > 0$, be the heat kernel on $\mathbb{T}$, i.e. let the heat semigroup be $H_Rf:=e^{-R\mathcal{L} }f=f\ast h_R$, $f \in L^2 (\mathbb{T})$. Since by \cite[Theorem 2.2]{COWMESE}, it holds $\|e^{-R\mathcal{L}}\|_{L^2(\mathbb{T})\rightarrow L^2(\mathbb{T})} =e^{-(1-\gamma(0))R} \leq 1$, by \cite[Chapter III, MAXIMAL THEOREM]{STEINred}, the heat maximal operator $f\rightarrow H_{\ast}f:=\sup_{R>0}|f\ast h_R|$ is bounded on $L^2 (\mathbb{T})$. Thus, it suffices to prove the boundedness of $(S^z-H)_{\ast}$. Using the Mellin transform for the functions $(1-t/R)^z_+$ and $e^{-tR}$, $R>0$, \cite{BertrandMellin}, and the spectral theorem for $\mathcal{L}$, we have
	\[\
	(S^z_R-H_R)f=\int_{\mathbb{R}}c(z,s)R^{-is}(\mathcal{L})^{is}fds, 
	\]
	where $|c(z,s)|\leq c(z)(1+|s|)^{-(\operatorname{Re}z+1)}$, \cite{GIUMA}, thus the integral above converges. Since $L^2(\mathbb{T})$ is a complete Banach lattice, from \cite{COW}, we can write
	\[\
	(S^z-H)_{\ast}f=\sup_{R>0}|(S_R^z-H_R)f|\leq c(z) \int_{\mathbb{R}} (1+|s|)^{-(\operatorname{Re}z+1)}|\mathcal{L}^{is}f|ds.
	\]
	Thus, 
	\[\
	\|(S^z-H)_{\ast}f\|_{L^2(\mathbb{T})}\leq c(z)\|f\|_{L^2(\mathbb{T})},
	\]
	since we have $\|\mathcal{L}^{is}\|_{L^2(\mathbb{T})\rightarrow L^2(\mathbb{T})}\leq 1$ by the spectral theorem and the fact that $\mathcal{L}$ is self-adjoint on $L^2(\mathbb{T})$ \cite[p.4271]{COWMESE}.
\end{proof}

\begin{lemma}\label{pr}
Let $q>2$. If $f\in L^p(\mathbb{T})$, $p\in [1, q^{\prime}]$ and $z\in \mathbb{C}$ with $\operatorname{Re}z>0$, then $\|S_{\ast}f\|_r\leq c \|f\|_p$, for $r\in [qp^{\prime}/(p^{\prime}-q), \infty]$.
\end{lemma}
\begin{proof} Since the kernel $\kappa_{R}^z$ is radial, we have by (\ref{convrad}) and Lemma \ref{kernelest} that
\begin{align*}
|S_R^zf(x)|&=|(f\ast \kappa_{R}^z)(x)|=|\sum\limits_{n\geq 0}\kappa_{R}^z(n)\sum\limits_{y\in S(x,n)}f(y)|\\
&\leq c\sum\limits_{n\geq 0}Q^{-n/2}\sum\limits_{y\in S(x,n)}|f(y)|\\
&=c(|f|\ast\kappa_Q)(x),
\end{align*}
where $\kappa_Q(n)=Q^{-n/2}$. Thus, 
\begin{equation}\label{maximalS}
|S_\ast^zf(x)|=\sup_{R>0}|(f\ast \kappa_{R}^z)(x)|
\leq c(|f|\ast\kappa_Q)(x),
\end{equation}
where $\kappa_Q(n)=Q^{-n/2}$. Since the kernel $\kappa_Q$ is radial, as in \cite[p.787]{EDDI}, we have
\begin{equation*}
\|\kappa_Q\|^q_{q}=|\kappa_Q(0)|^q+(Q+1)\sum_{n=1}^{\infty}Q^{n-1}|\kappa_R^z(n)|^q.
\end{equation*}
Thus, for every $q>2$, we have
\begin{align}\label{kQest}
\|\kappa_Q\|^q_{q}
&\leq c+c(Q+1)\sum_{n=1}^{\infty}Q^{n-1}Q^{-\frac{n}{2}q} \notag\\
&\leq c+c\frac{Q+1}{Q}\sum_{n=1}^{\infty}Q^{-\frac{n}{2}(q-2)}<c.
\end{align}
From (\ref{maximalS}) and (\ref{kQest}), for fixed $q> 2$ and $\operatorname{Re}z>0$, the operator $ S_\ast^{z}$ maps $
	L^p(\mathbb{T})$, $p\in[1, q^{\prime}]$, continuously to $L^r(\mathbb{T})$, for every $r\in
	[qp^{\prime}/ (p^{\prime}-q),\infty]$. \end{proof}   From Lemmata \ref{S*L2} and \ref{pr} and Riesz-Thorin interpolation, we obtain that for every $r\geq
pq/(2-p+pq-q)$, there is $c\left( z\right) >0$ such that for every $f\in
L^{p}(\mathbb{T})$, $1\leq p \leq 2$, 
\begin{equation*}
\Vert S_{\ast }^{z}f\Vert _{r}\leq c(z)\Vert f\Vert _{p}.
\end{equation*}
\textit{Step 2:  Theorem \ref{Thm} holds for a dense subspace $D$ of $L^p(\mathbb{T})$, $1\leq p \leq 2$.} 

We shall  modify the proof of \cite[Theorem 4]{ALEXOLO}. Set
\[
D=\{ e^{-\frac{1}{t}\mathcal{L}}e^{-s\mathcal{L}}f; \; f\in C_0^{\infty}(\mathbb{T}),\; t\geq 1, \;0<s\leq 1\}, 
\]
and recall that the heat operator is $L^p(\mathbb{T})$ bounded for every $1\leq p\leq 2$, \cite{COWMESE}. Note thus that $\| e^{-s\mathcal{L}}f-f\|_p \rightarrow 0$ as $s\rightarrow 0$ for all $f\in C_0^{\infty}(\mathbb{T})$ and $1\leq p \leq 2$. Also, the multiplier
\[
e^{-\frac{1}{t}(1-\gamma(\lambda))}e^{-s(1-\gamma(\lambda))}-e^{-s(1-\gamma(\lambda))}
\]
yields an $L^p$ bounded convolution operator, for every $1\leq p\leq 2$. It follows from the triangle inequality that
\begin{align*}
&\|e^{-\frac{1}{t}\mathcal{L}}e^{-s\mathcal{L}}f-f\|_p\rightarrow 0, \text{ as }t\rightarrow \infty.
\end{align*}
Thus,  the space $D$ is dense to all $L^p(\mathbb{T})$, $1\leq p\leq 2$.

Let us now fix some $g=e^{-\frac{1}{t}\mathcal{L}}e^{-s\mathcal{L}}f \in D$. Let us also consider a function $\psi \in C^{\infty}(\mathbb{R})$ such that 
\[
\psi(\lambda)=
\begin{cases} 1, \quad \text{for } |\lambda|\leq 1/4 \\
	0, \quad \text{for } |\lambda|\geq 1/2, 
\end{cases}
\]
and put $\psi_R(\lambda)=\psi(\lambda /R)$, $R>0$. Then, for $R$ large enough we have that 
\[
m_R^z(\mathcal{L})g=\psi_R(\mathcal{L})m_R^z(\mathcal{L})e^{-\frac{1}{t}\mathcal{L}}e^{-s\mathcal{L}}f
\] and therefore
\begin{equation}\label{**}
m_R^z(\mathcal{L})g-g= [\psi_R(\mathcal{L})m_R^z(\mathcal{L})-1]e^{-\frac{1}{t}\mathcal{L}}e^{-s\mathcal{L}}f.
\end{equation}
Thus, 
\begin{align*}
|m_R^z(\mathcal{L})g(x)-g(x)|&= \left|[\psi_R(\mathcal{L})m_R^z(\mathcal{L})-1]e^{-\frac{1}{t}\mathcal{L}}e^{-s\mathcal{L}}f(x)\right|\\
&\leq \| [\psi_R(\mathcal{L})m_R^z(\mathcal{L})-1]e^{-\frac{1}{t}\mathcal{L}}h_s(x,\cdot) \|_2 \|f\|_2\\
&\leq \sup_{\lambda>0}[\psi_R(\mathcal{\lambda})m_R^z(\mathcal{\lambda})-1]e^{-\frac{1}{t}(1-\gamma(\lambda))}\|h_s(x, \cdot)\|_2\|f\|_2,
\end{align*}
where the last inequality follows from the spectral theorem. Since 
\[
\sup_{\lambda>0}[\psi_R(\mathcal{\lambda})m_R^z(\mathcal{\lambda})-1]e^{-\frac{1}{t}(1-\gamma(\lambda))} \rightarrow 0, \text{ as } R\rightarrow \infty, 
\]
we obtain 
\[
|m_R^z(\mathcal{L})g(x)-g(x)| \rightarrow 0, \text{ as } R\rightarrow \infty,
\] 

\textit{Step 3: Theorem \ref{Thm} holds on the whole class $L^p(\mathbb{T})$}, $1\leq p \leq 2$.

As in \cite{ALEXOLO}, it suffices to combine Steps 1 and 2, and  well-known measure theoretic arguments, see for example \cite[Theorem 2.1.14]{GRAF}, and the proof of Theorem \ref{Thm} is complete. Note that since the measure on homogeneous trees is discrete, almost everywhere convergence reduces to everywhere convergence.

\end{document}